\documentclass[11pt]{amsart}

\usepackage[margin=1.1in]{geometry} %
\usepackage{amsmath, amssymb, amsthm, xcolor}
\usepackage{microtype}
\usepackage[hidelinks]{hyperref}

\theoremstyle{plain}

\newtheorem{theorem}{Theorem}[section]
\newtheorem{lemma}[theorem]{Lemma}
\newtheorem{corollary}[theorem]{Corollary}

\theoremstyle{definition}

\DeclareMathOperator{\TV}{TV}

\DeclareMathOperator{\Bin}{Bin}

\DeclareMathOperator{\Ber}{Ber}

\DeclareMathOperator{\Law}{\mathcal{L} }
\DeclareMathOperator{\Var}{Var}

\newcommand{\R}{\mathbb{R}}

\newcommand{\N}{\mathbb{N}}
\newcommand{\Z}{\mathbb{Z}}

\newcommand{\E}{\mathbb{E}}
\renewcommand{\P}{\mathbb{P}}

\newcommand{\beq}{\begin{eqnarray*}}
\newcommand{\eeq}{\end{eqnarray*}}
\newcommand{\beqn}{\begin{eqnarray}}
\newcommand{\eeqn}{\end{eqnarray}}

\newcommand{\eqdef}{:=}

\newcommand{\paren}[1]{\left( #1 \right)}

\newcommand{\nrm}[1]{\left\Vert #1 \right\Vert}

\newcommand{\set}[1]{\left\{ #1 \right\}}

\newcommand{\eps}{\varepsilon}
\newcommand{\mypoi}[1]{S_{#1}}

\newcommand{\mybin}[2]{\Bin(#1,#2)}

\newcommand{\etaBCV}{\eta_0}
\newcommand{\CBCV}{C_0}

\newcommand{\hide}[1]{}

\renewcommand{\vec}[1]{\textup{{\bf #1}}}

\title[TV homogenization inequalities]{TV homogenization inequalities}

\author[A.~Kontorovich]{Aryeh Kontorovich}
\address{Ben-Gurion University of the Negev}
\email{karyeh@cs.bgu.ac.il}

\date{} %

\begin{document}

\begin{abstract}
We study the total variation distance under two information-erasing maps on inhomogeneous Bernoulli product measures: summation and homogenization. While summation is a Markov kernel and hence satisfies the usual data processing inequality, homogenization --- which maps each Bernoulli parameter to the cumulative mean --- is not. Nevertheless, we prove that the homogenization map also reduces the TV distance, up to a universal constant. The argument is based on an explicit two-sided control of the TV distance between Poisson binomials, obtained via a parameter interpolation and a second-moment extraction.
\end{abstract}

\maketitle

\section{Introduction}

For $\vec{p}=(p_1,\dots,p_n)\in[0,1]^n$,
the
\emph{Poisson--binomial}
law, introduced by
\cite{Poisson1837}\footnote{
See \cite{PoissonSheynin2019} for a modern English translation.
},
is
the pushforward of the product measure
$
\Ber(\vec{p})
:=
\Ber(p_1)\otimes\cdots\otimes\Ber(p_n)$ under the sum map
$x\mapsto\sum_i x_i$.
Equivalently, if
$X_i\sim\Ber(p_i)$ are independent,
then
$S_{\vec{p}}:=\sum_{i=1}^n X_i$
is distributed according to this law.
The sum map is easily seen to be a Markov kernel,
and hence the Data Processing Inequality
\cite[Theorem 7.4]{Polyanskiy_Wu_2024}
applies:
$
\TV(\Ber(\vec{p}),\Ber(\vec{q}))
\ge
\TV(S_{\vec{p}},S_{\vec{q}})
$, where $\TV$ is the total variation distance.
Intuitively, since information is lost under the sum map,
the distributions can only become closer to each other.
In this paper, we investigate
the behavior of TV under
the sum map at the {\em parameter} rather than the {\em observation} level.
The {\em homogenization} map replaces each $p_i$
with
$\bar p=\frac1n\sum_i p_i$,
transforming $\Ber(\vec{p})$
into
$
\Ber(\bar p)^{\otimes n}
$.
Since this operation also erases information,
we might likewise expect it to decrease the TV distance.
One immediate obstruction is that
the homogenization map
cannot be realized by any Markov kernel.
Indeed,
suppose to the contrary that
for $n\ge2$,
some Markov kernel $K$
maps $X\sim \Ber(\vec{p})$ to
$
Y\sim
\Ber(\bar p)^{\otimes n}
$.
Applying the sum map to $Y$
yields
$Z\sim
\mybin{n}{\bar p}
$;
thus we have some Markov kernel $\tilde K:
\Ber(\vec{p})
\mapsto
\mybin{n}{\bar p}
$
for all $\vec{p}\in[0,1]^n$.
Consider a $\vec{p}$ of the form $\vec{p}=(t,0,0,\ldots,0)\in[0,1]^n$.
Then the law of $X$ is affine in $t$ (it is a $t$-convex combination of two atoms), as is the pushforward measure
under $\tilde K$.
However, $
\P(Z=2)=\binom{n}{2}\paren{\frac{t}{n}}^2\paren{1-\frac{t}{n}}^{n-2}
$, which is not affine in $t$.\footnote{
For the special case of {\em complementary} $\vec p,\vec q$,
meaning that $\vec q=\vec 1-\vec p$, it turns out that the homogenization operator can be realized as a Markov kernel;
see Appendix.
}

This indicates that the standard information-theoretic argument for proving
data processing inequalities does not apply, necessitating novel tools.
One such tool we propose
analyzes
the behavior of
the Poisson binomial
$S_{\vec{p}}$
under the total variation metric $\TV$.
For $\vec{p},\vec{q}\in[0,1]^n$, define the quantities
\beqn
\label{eq:delsigphi}
\Delta:=
\sum_{i=1}^n|p_i-q_i|
,
\quad
\sigma_{\vec{p}}^2:=\Var(S_{\vec{p}})=\sum_{i=1}^n p_i(1-p_i),
\quad
\Phi(\vec{p},\vec{q}):=
\min\!\left(1,\frac{\Delta}{\sqrt{\sigma_{\vec{p}}^2+1}}\right)
.
\eeqn
We discover
a simple analytic approximation of $\TV(S_{\vec{p}},S_{\vec{q}})$
for {\em dominating pairs} $\vec{p},\vec{q}\in[0,1]^n$, $p_i\ge q_i$,
up to universal multiplicative constants:
\beqn
\label{eq:TV-Phi}
\TV(S_{\vec{p}},S_{\vec{q}})\asymp\Phi(\vec{p},\vec{q}),
\qquad \vec{p},\vec{q}\in[0,1]^n, \vec{p}\ge \vec{q}.
\eeqn
For general $\vec{p},\vec{q}\in[0,1]^n$,
only the upper bound
$\TV(S_{\vec{p}},S_{\vec{q}})\lesssim\Phi(\vec{p},\vec{q})$ holds, which
can be sharpened to $\min(\Phi(\vec{p},\vec{q}),\Phi(\vec{q},\vec{p}))$ by symmetry.
Unlike
$\TV(\Ber(\vec{p}),\Ber(\vec{q}))$,
whose exact computation
is infeasible (see below),
$\TV(S_{\vec{p}},S_{\vec{q}})$
is computable
exactly
in time $O(n^2)$ via a simple recurrence
\cite[Section 5]{TangTang2023}\hide
{
One initializes
$f_0(0)=1$ and $f_0(k)=0$ for $k\neq 0$
and updates
\[
f_i(k)=(1-p_i)\,f_{i-1}(k)+p_i\,f_{i-1}(k-1),
\]
where by convention $f_{i-1}(-1)=0$ and $f_{i-1}(i)=0$.
After $n$ steps, $f_n(k)=\P(S_{\vec{p}}=k)$. An analogous calculation
is carried out for $S_{\vec{q}}$, from which $TV(S_{\vec{p}},S_{\vec{q}})$ is trivial to compute.
}.
Still,
one should not expect any analytically tractable analogue of
\eqref{eq:TV-Phi} for general $\vec{p},\vec{q}$ due to cancellations and convolutional smoothing.
As an application of these results, we prove the
TV homogenization inequality for all $\vec{p},\vec{q}\in[0,1]^n$:
\beqn
\label{eq:tv-homog}
\TV(\Ber(\vec{p}),\Ber(\vec{q}))
&\gtrsim&
\TV(
\Ber(\bar p)^{\otimes n}
,
\Ber(\bar q)^{\otimes n}
)
=
\TV(\mybin{n}{\bar p},\mybin{n}{\bar q})
.
\eeqn

\medskip\noindent\textbf{Notation.}
We write $[n]:=\set{1,\ldots,n}$.
For two distributions $P,Q$ on
a finite set
$\Omega$, their total variation
distance is defined by
\beqn
\label{eq:tv-def}
\TV(P,Q)
&=&
\frac12\sum_{\omega\in\Omega}|P(\omega)-Q(\omega)|
.
\eeqn
If $X,Y$ are random variables with laws $P,Q$, respectively,
then we write $\TV(X,Y):=\TV(P,Q)$.
For $p\in[0,1]$,
$\Ber(p)$
denotes
the Bernoulli measure on $\set{0,1}$:
$\Ber(p)(0)=1- p$
and
$\Ber(p)(1)=p$.
For $n\in\N$ and $\vec{p}
\in[0,1]^n$,
$\Ber(\vec{p})$ denotes
the product of $n$ Bernoulli distributions with parameters $p_i$:
$
\Ber(\vec{p}) = \Ber(p_1) \otimes \Ber(p_2)\otimes \ldots \otimes \Ber(p_n).
$
The Poisson binomial $S_{\vec{p}}$ is a sum of $n$ independent $\Ber(p_i)$
variables.
When all of the $p_i$ are identical (say, to $\bar p$),
this is the binomial distribution $\mybin{n}{\bar p}$.
The notation in \eqref{eq:delsigphi}
will persist throughout the paper.

\medskip\noindent\textbf{Main results.}
Our first result implies the upper bound in \eqref{eq:TV-Phi}.
Let $\etaBCV$ denote the sharp universal constant from
\cite[Theorem~1]{BAILLON_COMINETTI_VAISMAN_2016},
and set
\[
\CBCV \eqdef \sqrt{\frac54+\etaBCV^2}.
\]
Numerically (see \cite[Remark~1]{BAILLON_COMINETTI_VAISMAN_2016}),
$\etaBCV\approx 0.4688223555$ and hence $\CBCV\approx 1.2123507747$.
\begin{theorem}
\label{thm:PB-upper-Phi}
For $\vec{p},\vec{q}\in[0,1]^n$
and $\Phi$ as in \eqref{eq:delsigphi},
\beq
\TV(S_{\vec{p}},S_{\vec{q}})
&\le&
2\CBCV\,\min(
\Phi(\vec{p},\vec{q})
,
\Phi(\vec{q},\vec{p})
)
.
\eeq
\end{theorem}
L. Mattner \cite{Mattner2026}
has brought to our attention
the result of \cite{Tasto2015Proximitat},
which upper bounds
$\TV(S_{\vec{p}},S_{\vec{q}})$
by
\beq
\frac{5|
\E S_{\vec{p}}
-
\E S_{\vec{q}}
|}{
\sqrt{
\sigma^2_{\vec{p}}+\sigma^2_{\vec{q}}
}}
+
\frac{12|
\E S_{\vec{p}}
-
\E S_{\vec{q}}
|}{{1+
\sigma^2_{\vec{p}}+\sigma^2_{\vec{q}}
}}.
\eeq
This estimate has the attractive property
of being topologically equivalent to $\TV(S_{\vec{p}},S_{\vec{q}})$
in the sense that the two converge to $0$
on precisely the same set of sequence pairs.
It is currently an open question whether a
version of our $\Phi$ functional could be made
topologically equivalent to $\TV$ as well.

The technical core of the paper is a matching lower bound
for dominating pairs:
\begin{theorem}
\label{thm:PB-lower-Phi}
For $\vec{p},\vec{q}\in[0,1]^n$
satisfying $p_i\ge q_i$, $i\in[n]$,
and $\Phi$ as in \eqref{eq:delsigphi},
\beq
\TV(S_{\vec{p}},S_{\vec{q}})
&\ge&
\frac{1}{9}
\Phi(\vec{p},\vec{q})
.
\eeq
\end{theorem}
\noindent The assumption $\vec{p}\ge \vec{q}$ is essential for the lower bound,
as illustrated by the example $\vec p,\vec q\in[0,1]^n$, for even
$n$, given by
$
\vec p
=
\Bigl(\tfrac12,\tfrac12,\dots,\tfrac12\Bigr)
$,
$
\vec q
=
\Bigl(\tfrac12+\delta,\dots,\tfrac12+\delta,\ \tfrac12-\delta,\dots,\tfrac12-\delta\Bigr),
$
with $n/2$ copies of $\tfrac12+\delta$ followed by $n/2$ copies of $\tfrac12-\delta$.
The choice
$\delta=1/\sqrt n$
yields
$\Phi(\vec p,\vec q)=1$,
but $\TV(S_{\vec p},S_{\vec q})\asymp1/n$.
In general, for any $\vec{p},\vec{q}\in[0,1]^n$
and any partition $I,J$ of $[n]$,
the
trivial estimate via the partitioned pairs
$
\TV(S_{\vec{p}},S_{\vec{q}})
\le
\TV(S_{\vec{p}_I},S_{\vec{q}_I})
+
\TV(S_{\vec{p}_J},S_{\vec{q}_J})
$
holds.
However, this method cannot yield a matching lower bound,
as illustrated by the example
$a,b=\tfrac12\pm\eps$,
$
\vec{p}=(a,a,b,b)
$,
$
\vec{q}=(a,b,a,b)
$
with the partition
$I=\set{1,2}$, $J=\set{3,4}$.
Although
for general $\vec{p},\vec{q}$,
the approach of Theorem~\ref{thm:PB-lower-Phi}
does not lower-bound $\TV(S_{\vec{p}},S_{\vec{q}})$, it does
yield a lower bound
for
the larger quantity
$
\TV(\Ber(\vec{p}),\Ber(\vec{q}))
$:
\begin{theorem}
\label{thm:tv-ber-lb}
For $\vec{p},\vec{q}\in[0,1]^n$,
let $I=\set{i\in[n]:p_i\ge q_i}$ and $J=[n]\setminus I$,
with
$\vec{p}_I,\vec{p}_J,\vec{q}_I,\vec{q}_J$ the corresponding subsequences of $\vec{p},\vec{q}$.
Then
\beq
\TV(\Ber(\vec{p}),\Ber(\vec{q}))
&\ge&
\max\paren{
\TV(\Ber(\vec{p}_I),\Ber(\vec{q}_I))
,
\TV(\Ber(\vec{p}_J),\Ber(\vec{q}_J))
}\\
&\ge&
\frac{1}{9}
\max(
\Phi(\vec{p}_I,\vec{q}_I)
,
\Phi(\vec{q}_J,\vec{p}_J)
).
\eeq
\end{theorem}

The last piece to complete our
program is a
homogenization result for binomials:
\begin{lemma}
\label{lem:bin-homog}
For $n\ge1$, $\vec{p},\vec{q}\in[0,1]^n$, and $\emptyset\neq A\subset N:=
[n]
$, define
$\bar p_A=|A|^{-1}\sum_{i\in A} p_i$ and $\bar q_A$
analogously;
put
\beq
\delta_A &:=& \TV(\mybin{|A|}{\bar p_A},\mybin{|A|}{\bar q_A}).
\eeq
If $I,J$ form a partition of $N$ with $I,J\neq\emptyset$, then
\[
\delta_{N} \;\le\; 2\bigl(\delta_I+\delta_J\bigr).
\]
\end{lemma}

The homogenization inequality
claimed in \eqref{eq:tv-homog}
then follows:
\begin{theorem}
\label{thm:main-homog}
For all $\vec{p},\vec{q}\in[0,1]^n$, we have
\beq
\TV(\Ber(\vec{p}),\Ber(\vec{q}))
&\ge&
c\TV(\mybin{n}{\bar p},\mybin{n}{\bar q}),
\eeq
where $c\ge\frac{1}{72\CBCV}\approx
0.0115
$ is a universal constant.
\end{theorem}

\medskip\noindent\textbf{Remark.}
As the remark preceding Theorem~\ref{thm:PB-TV-symmetric-BCV} indicates, a sharpened constant $c\ge\frac{1}{36\CBCV}$
is straightforward to obtain via analogous arguments.
The choice
$\vec p = (1-2\varepsilon, \tfrac12)$,
$\vec q = (1, \tfrac12+\varepsilon)$,
$\eps\to0$ demonstrates that
it is possible for homogenization to (slightly) increase the TV distance,
and also that
$8/9$ is an upper bound on $c$.
We conjecture that this value is in fact optimal, see below.
\begin{proof}
Partition $[n]$ into $I,J$ as in
Theorem~\ref{thm:tv-ber-lb} and ignore the trivial case where one of $I,J$ is empty. Then
Theorem~\ref{thm:tv-ber-lb}
implies
\beq
\TV(\Ber(\vec{p}),\Ber(\vec{q}))
&\ge&
\frac{1}{9}
\max(
\Phi(\vec{p}_I,\vec{q}_I)
,
\Phi(\vec{q}_J,\vec{p}_J)
)
\\
&\ge&
\frac{1}{18}\paren{
\Phi(\vec{p}_I,\vec{q}_I)
+
\Phi(\vec{q}_J,\vec{p}_J)
}
.
\eeq
A simple convexity argument shows that
homogenization cannot increase $\Phi$:
\beq
\Delta(\vec{p},\vec{q}) &\ge& \Delta(\bar p\boldsymbol{1},\bar q\boldsymbol{1}),\\
\sigma_{\vec{p}}^2 &\le& \sigma_{\bar p\boldsymbol{1}}^2=n\bar p(1-\bar p),
\eeq
where $\bar p\boldsymbol{1}=(\bar p,\bar p,\ldots, \bar p)
$
is of the same dimension as $\vec{p}$.
Now homogenize
$\bar p_I$,
$\bar p_J$,
$\bar q_I$,
$\bar q_J$
as in Lemma~\ref{lem:bin-homog}.
Since $
\mybin{|I|}{\bar p_I}
$ is a special case of Poisson binomial,
Theorem~\ref{thm:PB-upper-Phi} applies:
\beq
\Phi(\vec{p}_I,\vec{q}_I)
+
\Phi(\vec{q}_J,\vec{p}_J)
&\ge&
\Phi(\bar p_I\boldsymbol{1},\bar q_I\boldsymbol{1})
+
\Phi(
\bar q_J\boldsymbol{1}
,
\bar p_J\boldsymbol{1}
)
\\
&\ge&
\frac{1}{2\CBCV}
\paren{
\TV(
\mybin{|I|}{\bar p_I}
,
\mybin{|I|}{\bar q_I}
)
+
\TV(
\mybin{|J|}{\bar p_J}
,
\mybin{|J|}{\bar q_J}
)
}\\
&\ge&
\frac{1}{
4
\CBCV}
\TV(
\mybin{n}{\bar p}
,
\mybin{n}{\bar q}
),
\eeq
where the
last
inequality is by
Lemma~\ref{lem:bin-homog}.
Finally, the identity
$
\TV(
\Ber(\bar p)^{\otimes n}
,
\Ber(\bar q)^{\otimes n}
)$
$=$
$\TV(\mybin{n}{\bar p},\mybin{n}{\bar q})
$
is a standard consequence of the
Neyman–Pearson lemma
(the likelihood ratio is determined by
and monotone in
the sum).
\end{proof}

\medskip\noindent\textbf{Remark and open problems.}
Although the homogenization inequality
relies on novel structural insights into the Poisson binomial,
there are compelling reasons to believe that a great deal
more structure remains to be uncovered, currently out of reach.
Indeed, extensive numerical experiments suggest that
the correct constant in Theorem~\ref{thm:main-homog} should be
$c=8/9$
and also that the bound in Lemma~\ref{lem:bin-homog}
can be sharpened to $\delta_N\le \delta_I+\delta_J-\delta_I\delta_J$ (the latter, in particular, appears deceptively simple, since only binomials are involved).
Our present methods do not seem to provide
any pathway to these conjecturally optimal bounds, which will have to
await further structural advances.

\medskip\noindent\textbf{Related work.}
This paper follows a program initiated by \cite{kon25tens}
to provide simple, analytically tractable upper and lower
estimates on $\TV(\Ber(\vec{p}),\Ber(\vec{q}))$
in terms of the $\vec{p},\vec{q}$;
the latter's contribution sharpened the trivial
lower bound $\TV(\Ber(\vec{p}),\Ber(\vec{q}))\ge\nrm{\vec{p}-\vec{q}}_\infty$
to $
\gtrsim\nrm{\vec{p}-\vec{q}}_2$.
We note that
generally, this result is incomparable with
the lower bound
$
\TV(\Ber(\vec{p}),\Ber(\vec{q}))$
$\gtrsim$
$\Phi(\vec{p},\vec{q})
$
implied by Theorem~\ref{thm:PB-lower-Phi}
for $\vec p\ge\vec q$.
Indeed,
the choice $q_i\equiv\frac12$
for $i\in[n]$ and $p_i=1/2$, $i>1$ and $p_1=\frac12+\eps$
yields $\Phi(\vec{p},\vec{q})\asymp\eps/\sqrt{n}$ while $\nrm{\vec{p}-\vec{q}}_2=\eps$.
For the other direction, the choice $\vec{q}\equiv0$ and $\vec{p}\equiv1/n$
yields
$\Phi(\vec{p},\vec{q})\asymp1$ and $\nrm{\vec{p}-\vec{q}}=n^{-1/2}$.
This program is continued in \cite{KonAv26},
where
an analytical closed-form
$O(\sqrt{\log n})$-factor
approximation
to
$
\TV(\Ber(\vec{p}),\Ber(\vec{q}))
$
is obtained.

Homogenization inequalities appear not to have been widely studied;
one classic result due to
Hoeffding \cite{Hoeffding1956}
is that homogenization under a fixed-mean constraint maximizes $\E g(S_{\vec{p}})$
for any convex $g$.
The study of
$S_{\vec{p}}$
is amply motivated by numerous applications
and
a venerable line of work
has considered
approximating this distribution
by simpler ones
\cite{TangTang2023};
perhaps most famous is Le Cam's inequality $\TV(S_{\vec{p}},\mathrm{Poi}(\sum p_i))\le \sum{p_i^2}$ \cite{MR142174}.
Note, however, that the latter yields an additive, rather than multiplicative approximation to $\TV(S_{\vec{p}},S_{\vec{q}})$.
Binomial approximation bounds in strong metrics (including $\TV$) appear in
\cite{Ehm1991} and refinements based on orthogonal-polynomial expansions and asymptotics in
\cite{Roos2000}, then generalized
to arbitrary laws \cite{Roos2010ClosenessConvolutions}
and also gave an additive homogenization-type estimate.
Later, \cite{Roos2017RefinedTV}
obtained TV bounds for approximating general convolutions
by compound Poisson laws using Kerstan's method together with novel ``smoothness inequalities''.
A delicate structural result
due to
\cite{DaskalakisPapadimitriou2015SparseCovers}
(also ultimately yielding an additive estimate)
shows that
every Poisson binomial admits an $\eps$-approximation
in TV by a distribution of one of two canonical ``compressed'' types: either
a \emph{sparse} $S_{\vec{p}}$ with only $O(1/\eps^3)$ nontrivial summands, or a
\emph{near-binomial} form.
Of particular relevance to this work is \cite{BAILLON_COMINETTI_VAISMAN_2016, BarbourJensen1989LocalTailPoisson}, whose structural results we build upon in Theorem~\ref{thm:PB-upper-Phi}.
More generally, $\TV(\Ber(\vec{p}),\Ber(\vec{q}))$ admits classic approximations in terms of more analytically tractable
proxies such as KL-divergence and the Hellinger distance;
their properties and limitations are discussed in \cite{kon25tens}.
Regarding the algorithmic aspect,
\cite{DBLP:conf/ijcai/0001GMMPV23}
showed
that computing
$\TV(\Ber(\vec{p}),\Ber(\vec{q}))$
exactly
for general $\vec{p},\vec{q}\in[0,1]^n$
is hard in the
$\#$P sense.
An efficient randomized algorithm guaranteeing a $1\pm\eps$
multiplicative approximation with confidence $1-\delta$, in time $O(\frac{n^2}{\eps^2}\log\frac1\delta)$
was discovered by
\cite{FengApproxTV23},
and later
derandomized by
\cite{Feng24Deterministically}.
\section{Proofs}

\subsection{Proof of Theorem \ref{thm:PB-upper-Phi}}
We will prove a more general, symmetric bound, which immediately
implies the one in Theorem \ref{thm:PB-upper-Phi}.
The constant $2\CBCV$ in the latter's bound could be sharpened by a factor of $2$ via a direct (asymmetric) proof.

\begin{theorem}
\label{thm:PB-TV-symmetric-BCV}
For $\vec{p},\vec{q}\in[0,1]^n$ and $S_{\vec{p}},S_{\vec{q}}$ the corresponding
Poisson binomials with variances $\sigma_{\vec{p}}^2$, $\sigma_{\vec{q}}^2$,
we have
\beq
\TV(S_{\vec{p}},S_{\vec{q}})
&\le&
\frac{2\CBCV\,\Delta}{\sqrt{\sigma_{\vec{p}}^2+1}+\sqrt{\sigma_{\vec{q}}^2+1}}.
\eeq
\end{theorem}

\begin{proof}
The proof uses three standard ingredients.

\noindent{\em Shift-TV for unimodal pmfs.}
If $Z$ is integer-valued with unimodal pmf $h(k)=\P(Z=k)$, then
\begin{equation}\label{eq:shift-TV-unimodal-thm}
\TV(Z,Z+1)=\max_k h(k),
\end{equation}
by a simple telescoping argument.

\noindent{\em Anti-concentration for Poisson binomials (Barbour-Jensen; Baillon--Cominetti--Vaisman).}
If $Z=\sum_{j=1}^m \Ber(r_j)$ has variance $v=\Var(Z)$, then
\cite[Theorem~1]{BAILLON_COMINETTI_VAISMAN_2016} proves the sharp bound
\begin{equation}\label{eq:pb-peak-bcv}
\max_k \P(Z=k)\le \frac{\etaBCV}{\sqrt{v}}
\qquad (v>0),
\end{equation}
where $\etaBCV$ is a universal constant
($\etaBCV\approx 0.4688223555$, \cite[Remark~1]{BAILLON_COMINETTI_VAISMAN_2016}).\footnote{
L. Mattner \cite{Mattner2026} informs us that
this result is contained in \cite{BarbourJensen1989LocalTailPoisson}
with a slightly worse constant.
}

In particular, since $\max_k\P(Z=k)\le 1$ always, we have for all $v\ge0$
\begin{equation}\label{eq:pb-peak-bcv-min}
\max_k \P(Z=k)\le \min\!\left\{1,\ \frac{\etaBCV}{\sqrt{v}}\right\}.
\end{equation}

\noindent{\em Unimodality of Poisson binomial pmfs.}
It is a classic
consequence of
the Aissen-Edrei-Schoenberg-Whitney
theorem \cite[Section 4]{TangTang2023}
that the Poisson binomial
has a log-concave,
and hence unimodal, law.
To combine these three ingredients, we begin with the interpolation.
For $t\in[0,1]$ define
\beq
r_i(t):=(1-t)q_i+tp_i,
\qquad
S(t):=\sum_{i=1}^n \Ber(r_i(t)),
\qquad t\in[0,1],
\eeq
where the Bernoullis are independent;
thus, $S(0)=S_{\vec{q}}$ and $S(1)=S_{\vec{p}}$ in distribution.
For $A\subseteq\mathbb Z$,
we write $f_A(t):=\P(S(t)\in A)$.
For each $i\in[n]$ let $X_i(t)\sim\Ber(r_i(t))$ denote the $i$th summand and set
\[
T_i(t):=S(t)-X_i(t)=\sum_{j\ne i}\Ber(r_j(t)).
\]
Then $T_i(t)$ is independent of $X_i(t)$ and $S(t)=T_i(t)+X_i(t)$.
Conditioning on $X_i(t)$ gives
\[
f_A(t)=r_i(t)\P(T_i(t)+1\in A)+(1-r_i(t))\P(T_i(t)\in A).
\]
View $f_A$ as a multilinear polynomial in the coordinates $(r_1,\dots,r_n)$.
By the chain rule,
$\frac{\mathrm{d}}{\mathrm{d}t} f_A(r(t))=\sum_{i=1}^n r_i'(t)\,\partial f_A/\partial r_i$,
and conditioning on $X_i(t)$ yields
$\partial f_A/\partial r_i=\P(T_i(t)+1\in A)-\P(T_i(t)\in A)$.
Differentiating and using $r_i'(t)=p_i-q_i$ yields
\[
f_A'(t)=\sum_{i=1}^n (p_i-q_i)\Bigl(\P(T_i(t)+1\in A)-\P(T_i(t)\in A)\Bigr).
\]
Hence, by the definition of total variation,
\[
|f_A'(t)|
\le
\sum_{i=1}^n |p_i-q_i|\,
\TV\!\bigl(T_i(t)+1,T_i(t)\bigr)
.
\]
Each $T_i(t)$ is Poisson binomial, hence unimodal, so by
\eqref{eq:shift-TV-unimodal-thm},
\[
\TV\!\bigl(T_i(t)+1,T_i(t)\bigr)
=\max_k \P(T_i(t)=k).
\]
By \eqref{eq:pb-peak-bcv-min},
\[
\max_k \P(T_i(t)=k)
\le
\min\!\left\{1,\ \frac{\etaBCV}{\sqrt{\Var(T_i(t))}}\right\}.
\]
Since $S(t)=T_i(t)+X_i(t)$ with independence and $\Var(X_i(t))\le \tfrac14$, we have
\[
\Var(T_i(t))=\Var(S(t))-\Var(X_i(t))\ge \Var(S(t))-\frac14.
\]
Let $x_+:=\max\{x,0\}$. Then
\[
\max_k \P(T_i(t)=k)
\le
\min\!\left\{1,\ \frac{\etaBCV}{\sqrt{\bigl(\Var(S(t))-\frac14\bigr)_+}}\right\}.
\]
We now upper bound this ``singular'' expression by a smooth envelope.
Set $\CBCV:=\sqrt{\tfrac54+\etaBCV^2}$. Then for every $x\ge0$ we have
\begin{equation}\label{eq:bcv-envelope}
\min\!\left\{1,\ \frac{\etaBCV}{\sqrt{(x-\frac14)_+}}\right\}
\le
\frac{\CBCV}{\sqrt{x+1}}.
\end{equation}
Indeed, if $x\le \tfrac14+\etaBCV^2$, the left-hand side equals $1$ and the right-hand side is
$\CBCV/\sqrt{x+1}\ge \CBCV/\sqrt{\tfrac54+\etaBCV^2}=1$.
If $x\ge \tfrac14+\etaBCV^2$, then the left-hand side equals $\etaBCV/\sqrt{x-\tfrac14}$ and
squaring shows that \eqref{eq:bcv-envelope} is equivalent to $x\ge \tfrac14+\etaBCV^2$.
Applying \eqref{eq:bcv-envelope} with $x=\Var(S(t))$ gives
\[
\TV\!\bigl(T_i(t)+1,T_i(t)\bigr)
\le
\frac{\CBCV}{\sqrt{\Var(S(t))+1}},
\]
and substituting back,
\begin{equation}\label{eq:fprime-bound}
|f_A'(t)|\le \frac{\CBCV\,\Delta}{\sqrt{\Var(S(t))+1}}.
\end{equation}
Since $u\mapsto u(1-u)$ is concave,
\[
r_i(t)(1-r_i(t))
\ge
(1-t)\,q_i(1-q_i)+t\,p_i(1-p_i),
\qquad
i\in[n],
\]
and summing over $i$ yields
\[
\Var(S(t))
\ge
(1-t)\sigma_{\vec{q}}^2+t\sigma_{\vec{p}}^2.
\]
Combining with \eqref{eq:fprime-bound},
\[
|f_A'(t)|
\le
\frac{\CBCV\,\Delta}{\sqrt{(1-t)\sigma_{\vec{q}}^2+t\sigma_{\vec{p}}^2+1}}.
\]
Integrating over $t$ gives
\[
\bigl|f_A(1)-f_A(0)\bigr|
\le
\CBCV\,\Delta\int_0^1 \frac{\mathrm{d}t}{\sqrt{(1-t)(\sigma_{\vec{q}}^2+1)+t(\sigma_{\vec{p}}^2+1)}}.
\]
Write $(1-t)(\sigma_{\vec{q}}^2+1)+t(\sigma_{\vec{p}}^2+1)=b+at$ with
$b=\sigma_{\vec{q}}^2+1$ and $a=\sigma_{\vec{p}}^2-\sigma_{\vec{q}}^2$.
If $a=0$, the integral is $1/\sqrt{b}$.
If $a\neq 0$, then
\[
\int_0^1 \frac{\mathrm{d}t}{\sqrt{b+at}}
=
\left[\frac{2}{a}\sqrt{b+at}\right]_{0}^{1}
=
\frac{2}{\sqrt{b+a}+\sqrt b}
=
\frac{2}{\sqrt{\sigma_{\vec{p}}^2+1}+\sqrt{\sigma_{\vec{q}}^2+1}}.
\]
Thus, for every event $A$,
\[
\bigl|\P(S_{\vec{p}}\in A)-\P(S_{\vec{q}}\in A)\bigr|
\le
\frac{2\CBCV\,\Delta}{\sqrt{\sigma_{\vec{p}}^2+1}+\sqrt{\sigma_{\vec{q}}^2+1}}.
\]
Taking $\sup_A$ proves the claim.
\end{proof}

\subsection{Proof of Theorem \ref{thm:PB-lower-Phi}}

We make use of a
concentration-variance inequality
stated in
\cite[Disp. (18)]{MattnerSchulz2018NormalApproximations}
and proved therein,
which immediately implies the following 
second-moment extraction
lemma. We thank 
L. Mattner
\cite{Mattner2026} for bringing this to our attention and improving our original
lower bound by a factor of $4/3$.
We also thank him for the
attributional note, ``This is due to Paul Lévy in a sharper version,
and was in essentially the present version apparently independently
discovered by \cite[Corollary 2.2]{BobkovChistyakov2015}''.

\begin{lemma}
\label{lem:pigeonhole}
For $g:\mathbb{Z}\to[0,\infty)$
and $\mu\in\R$,
define
\[
  G := \sum_{k\in\mathbb{Z}} g(k)
  \qquad
  J :=
  \sum_{k\in\mathbb{Z}} (k-\mu)^2\, g(k).
\]
If $0<G,J<\infty$, then
\[
  \sup_{k\in\mathbb{Z}} g(k)
  \;\ge\;
  \frac{G^{3/2}}{4\sqrt{J}+\sqrt{G}}.
\]
\end{lemma}
\begin{proof}
Normalize $g$ to be a probability measure $P$ via $P(\set{k})=g(k)/G$
and let $X$ be distributed according to $P$;
then
$Var(X)\le J/G$.
By 
Lévy's extraction principle,
\cite[Disp.~(18)]{MattnerSchulz2018NormalApproximations}, for any law
with variance $\sigma^2$ and any $h>0$,
\[
\sup_{x\in\mathbb{R}} P\bigl((x,x+h)\bigr)
\;\ge\;
\frac{h}{\sqrt{h^2+12\sigma^2}}.
\]
Apply this with $h=1$ and $\sigma^2=\Var(X)$. Since $P$ is supported on $\mathbb{Z}$,
\[
\sup_{x\in\mathbb{R}} P((x,x+1))
= \max_{k\in\mathbb{Z}} P(\{k\})
= \frac{1}{G}\max_{k\in\mathbb{Z}} g(k).
\]
Therefore,
\beq
\max_k g(k)
&\ge&
\frac{G}{\sqrt{1+12\Var(X)}}
\ge
\frac{G}{\sqrt{1+12(J/G)}}
=
\frac{G^{3/2}}{\sqrt{G+12J}}
.
\eeq
Finally, since $\sqrt{G+12J}\le \sqrt{G}+4\sqrt{J}$, we conclude
$
\max_{k\in\mathbb{Z}} g(k)
\;\ge\;
\frac{G^{3/2}}{\,4\sqrt{J}+\sqrt{G}\,}
$,
as claimed.
\end{proof}

As above, $(\vec{p},\vec{q})$ is a dominating pair with $p_i\ge q_i$ and associated
Poisson binomials
$S_{\vec{p}}$, $S_{\vec{q}}$.
We will apply the pigeonhole lemma to extract a ``wedge'' event
with sufficient separation in probability under
$S_{\vec{p}}$
and $S_{\vec{q}}$ to achieve the requisite TV lower bound.
To this end, we define the $G$ and $J$ functionals --- both
in terms of the non-negative (by the monotone coupling) measure $g$ on $\Z$:
\begin{equation}\label{eq:gdef}
  g(k) := \P(X\ge k)-\P(Y\ge k),
  \qquad k\in\mathbb{Z}.
\end{equation}
Then
$
G
  = \sum_{k\in\mathbb{Z}} g(k)
$
has the more familiar form
\[
  G:=\Delta = \mu_{\vec{p}}-\mu_{\vec{q}} = \E X-\E Y
  ,
\]
where
$  X := S_{\vec{p}},  Y := S_{\vec{q}}$
and
$  \mu_{\vec{p}} := \E X,
  \mu_{\vec{q}} := \E Y
  $.
Moreover, for any $k$,
\[
  |g(k)| = |\P(X\ge k)-\P(Y\ge k)|
  \le \TV(\mypoi{\vec{p}},\mypoi{\vec{q}}),
\]
and so
\begin{equation}\label{eq:TV-ge-supg}
  \TV(\mypoi{\vec{p}},\mypoi{\vec{q}})
  \;\ge\; \sup_{k\in\mathbb{Z}} g(k).
\end{equation}
Finally,
$  J = \sum_{k\in\mathbb{Z}} (k-\mu_{\vec{p}})^2 g(k)
$ is defined as in Lemma~\ref{lem:pigeonhole}.
The technical core of the argument hinges on upper bounding $J$:

\begin{lemma}
\label{lem:J-global}
Let $X,Y,\sigma_{\vec{p}}^2=\Var(X),g,\mu_{\vec{p}},\Delta,J$ be defined as above.
Then
\[
  J \;\le\;
2\,\Delta\,\bigl(\sigma_{\vec{p}}^2+1+\Delta^2\bigr).
\]
\end{lemma}

\begin{proof}
For $i\in[n]$,
write $\Delta_i:=p_i-q_i\ge0$ with $\sum_i\Delta_i=\Delta$. For
$t\in[0,1]$ set
\[
  r_i(t) := q_i + t\Delta_i,\quad
  X_i(t)\sim\Ber(r_i(t))\ \text{independent},\quad
  S(t):=\sum_{i=1}^n X_i(t),\quad
  T_i(t):=S(t)-X_i(t).
\]
Then $S(0)=Y$ and $S(1)=X$ in distribution. For fixed $n$, each
$S(t)$ has a distribution that is a finite polynomial in the parameters
$\{r_i(t)\}$, so for each $k$ and $i$ the function
\beq
F_k(t) &:=&
\P(S(t)\ge k) \\
&=&
r_i(t)\,\P(T_i(t)\ge k-1)\;+\;(1-r_i(t))\,\P(T_i(t)\ge k)
.
\eeq
is a polynomial in $t$ and hence differentiable on $[0,1]$.
Differentiating,
\[
  \frac{\partial}{\partial r_i}F_k
  = \P(T_i(t)\ge k-1)-\P(T_i(t)\ge k)
  = \P(T_i(t)=k-1).
\]
By the chain rule and the fact that $r_i'(t)=\Delta_i$,
\[
  \frac{\mathrm{d}}{\mathrm{d}t}\,\P(S(t)\ge k)
  = \sum_{i=1}^n \Delta_i\,\P\bigl(T_i(t)=k-1\bigr),
  \qquad k\in\mathbb Z.
\]
Integrating from $0$ to $1$ and using $S(0)=Y$,
$S(1)=X$ in distribution,
\[
  g(k)
  = \P(X\ge k)-\P(Y\ge k)
  = \int_0^1 \sum_{i=1}^n \Delta_i\,\P\bigl(T_i(t)=k-1\bigr)\,\mathrm{d}t,
  \qquad k\in\mathbb Z.
\]
In this formulation, the functional $J=\sum_k (k-\mu_{\vec{p}})^2 g(k)$
becomes
\begin{align*}
  J
  &= \int_0^1 \sum_{i=1}^n \Delta_i
       \sum_{k\in\mathbb Z} (k-\mu_{\vec{p}})^2 \P\bigl(T_i(t)=k-1\bigr)\,\mathrm{d}t.
\end{align*}
Changing variable $j=k-1$,
\[
  \sum_{k\in\mathbb Z} (k-\mu_{\vec{p}})^2 \P\bigl(T_i(t)=k-1\bigr)
  = \E_t\bigl[(T_i(t)+1-\mu_{\vec{p}})^2\bigr],
\]
where $\E_t$ denotes expectation under the product law with
parameters $\{r_j(t)\}_j$.
Hence
\begin{equation}\label{eq:J-Et}
  J
  = \int_0^1 \sum_{i=1}^n \Delta_i\,
       \E_t\bigl[(T_i(t)+1-\mu_{\vec{p}})^2\bigr]\,\mathrm{d}t.
\end{equation}
We bound $\E_t[(T_i(t)+1-\mu_{\vec{p}})^2]$ uniformly in $t,i$
via the decomposition
$\E[(Z-\alpha)]^2=\Var(Z)+(\E Z-\alpha)^2$,
where $Z=T_i(t)+1$ and $\alpha=\mu_{\vec{p}}$.
To bound the variance term, note that
$f(x)=x(1-x)$
is $1$-Lipschitz on $[0,1]$, and so
\[
  |f(r_j(t))-f(p_j)| \le |r_j(t)-p_j| = (1-t)\Delta_j.
\]
Hence
\[
  \Var_t(S(t))
  = \sum_j f(r_j(t))
  \le \sum_j f(p_j) + (1-t)\sum_j\Delta_j
  = \sigma_{\vec{p}}^2 + (1-t)\Delta
  \le \sigma_{\vec{p}}^2 + \Delta
  \le\sigma_{\vec{p}}^2 + 1 + \Delta^2.
\]
Since $T_i(t)=S(t)-X_i(t)$ with $X_i(t)$ Bernoulli,
\beqn
\label{eq:VarTi-bound}
  \Var_t(T_i(t))
  \le \Var_t(S(t))
  \le \sigma_{\vec{p}}^2 + 1 + \Delta^2.
\eeqn

To bound the bias term, we have
\[
  \E_t[S(t)] = \mu_{\vec{q}} + t\Delta,
\]
and $\mu_{\vec{p}}=\mu_{\vec{q}}+\Delta$, so $\E_t[S(t)]-\mu_{\vec{p}}=-(1-t)\Delta$.
Also $r_i(t)=q_i+t\Delta_i$, so
\[
  \E_t[T_i(t)]
  = \mu_{\vec{q}} + t\Delta - (q_i + t\Delta_i),
\]
and hence
\[
  M_i(t) := \E_t[T_i(t)+1-\mu_{\vec{p}}]
  = 1 - q_i - \Delta + t(\Delta-\Delta_i).
\]
It is straightforward to see that
$|M_i(t)| \le \max(1,\Delta)$
and thus,
\beq
  \E_t\bigl[(T_i(t)+1-\mu_{\vec{p}})^2\bigr]
  &=& \Var_t(T_i(t)) + M_i(t)^2
  \\
  &\le&
 (\sigma_{\vec{p}}^2+1+\Delta^2) + (\sigma_{\vec{p}}^2+1+\Delta^2)
  = 2(\sigma_{\vec{p}}^2+1+\Delta^2).
\eeq
Substituting into \eqref{eq:J-Et} gives
\[
  J
  \le \int_0^1 \sum_i \Delta_i\,
         2(\sigma_{\vec{p}}^2+1+\Delta^2)\,\mathrm{d}t
  = 2\,\Delta\,(\sigma_{\vec{p}}^2+1+\Delta^2),
\]
as claimed.
\end{proof}

\begin{corollary}
[Proof of Theorem \ref{thm:PB-lower-Phi}]
\label{cor:PB-global}
Let $(\vec{p},\vec{q})$ be a dominating pair with
laws $\mypoi{\vec{p}},\mypoi{\vec{q}}$,
sums $S_{\vec{p}},S_{\vec{q}}$, $\sigma_{\vec{p}}^2 := \Var(S_{\vec{p}})$,
and
$
  \Delta := \E S_{\vec{p}}-\E S_{\vec{q}}
$.
Then
\[
  \TV(\mypoi{\vec{p}},\mypoi{\vec{q}})
  \;\ge\;
  c\,
  \min\Bigl(1,\frac{\Delta}{\sqrt{\sigma_{\vec{p}}^2+1}}\Bigr),
\]
for some universal constant $c\ge\frac1{9}$.
\end{corollary}

\begin{proof}
Recall $g(k):=\P(S_{\vec{p}}\ge k)-\P(S_{\vec{q}}\ge k)\ge0$ and
\[
  \Delta = \sum_k g(k),\qquad
  J = \sum_k (k-\mu_{\vec{p}})^2 g(k),
  \qquad \mu_{\vec{p}}=\E S_{\vec{p}}.
\]
By \eqref{eq:TV-ge-supg},
$
  \TV(\mypoi{\vec{p}},\mypoi{\vec{q}})
  \ge \sup_{k\in\mathbb{Z}} g(k)$.
By Lemma~\ref{lem:J-global},
\[
  J \;\le\; 2\,\Delta\,(\sigma_{\vec{p}}^2+1+\Delta^2).
\]
Apply Lemma~\ref{lem:pigeonhole} with $G=\Delta$
to obtain
\[
  \sup_k g(k)
  \;\ge\;
  \frac{\Delta}{4\sqrt{2(\sigma_{\vec{p}}^2+1+\Delta^2)} + 1}.
\]
We compare $\sqrt{\sigma_{\vec{p}}^2+1+\Delta^2}$ with
$\sqrt{\sigma_{\vec{p}}^2+1}$ and consider the two cases.
If
$\Delta \le \sqrt{\sigma_{\vec{p}}^2+1}$,
then $\Delta^2\le\sigma_{\vec{p}}^2+1$, so
$
  \sigma_{\vec{p}}^2+1+\Delta^2 \le 2(\sigma_{\vec{p}}^2+1),
$
and
\begin{align*}
  \TV(\mypoi{\vec{p}},\mypoi{\vec{q}})
  &\ge \frac{\Delta}{4\sqrt{2}\sqrt{2(\sigma_{\vec{p}}^2+1)}+1}
  = \frac{\Delta}{8\sqrt{\sigma_{\vec{p}}^2+1}+1} %
\ge
\frac{\Delta}{9\sqrt{\sigma_{\vec{p}}^2+1}}
  .
\end{align*}
If
$\Delta \ge \sqrt{\sigma_{\vec{p}}^2+1}$,
then $\Delta^2\ge\sigma_{\vec{p}}^2+1$, so
$
  \sigma_{\vec{p}}^2+1+\Delta^2 \le 2\Delta^2,
$
and
\begin{align*}
  \TV(\mypoi{\vec{p}},\mypoi{\vec{q}})
  &\ge \frac{\Delta}{4\sqrt{2}\sqrt{2\Delta^2}+1}
  = \frac{\Delta}{8\Delta+1}
  \ge \frac19
  .
\end{align*}
Combining the two cases proves the claim.
\end{proof}

\subsection{Proof of Theorem \ref{thm:tv-ber-lb}}
Both inequalities invoke the data processing inequality;
the first in the form
\beq
\TV(P\otimes P',Q\otimes Q')
&\ge&
\max(
\TV(P,Q)
,
\TV(P',Q')
),
\eeq
and the second in the form
$
\TV(\Ber(\vec{p}),\Ber(\vec{q}))
\ge
\TV(S_{\vec{p}},S_{\vec{q}})
$.
Given these, the claim is
an immediate consequence of Theorem~\ref{thm:PB-lower-Phi}.

\subsection{Proof of Lemma \ref{lem:bin-homog}}
\newcommand{\nI}{|I|}
\newcommand{\nJ}{|J|}

\begin{proof}
We begin by representing $\mybin{n}{\bar p_N}$ as a mixture.
Let $M\sim\mybin{n}{w}$ where $w = \nI/n$. Conditionally on $M=m$, let
\[
U_m\sim\mybin{m}{\bar p_I},\qquad V_m\sim\mybin{n-m}{\bar p_J}
\]
be independent and set $S_m:=U_m+V_m$.
We claim that $S_M\sim\mybin{n}{\bar p_N}$.
Indeed, for $t\in\R$, put $s:=1-\bar p_I+\bar p_I t$ and $r:=1-\bar p_J+\bar p_J t$.
Compute the probability generating function conditional on $M$:
\[
\E[t^{S_M}\mid M]
=\E[t^{U_M+V_M}\mid M]
=\E[t^{U_M}\mid M]\;\E[t^{V_M}\mid M]
=s^{\,M}r^{\,n-M},
\]
so $\E[t^{S_M}]=\E[s^M r^{n-M}]$.
Represent $M=\sum_{i=1}^n X_i$ where
$X_1,\dots,X_n$ are i.i.d.\ $\Ber(w)$. Then $n-M=\sum_{i=1}^n(1-X_i)$ and
\[
s^M r^{n-M}
=\prod_{i=1}^n s^{X_i}r^{1-X_i}
=\prod_{i=1}^n\bigl(r+(s-r)X_i\bigr).
\]
Taking expectation and using independence,
\[
\E[s^M r^{n-M}]
=\prod_{i=1}^n \E\bigl[r+(s-r)X_i\bigr]
=\bigl(r+(s-r)\E[X_1]\bigr)^n
=\bigl((1-w)r+ws\bigr)^n.
\]
Substituting $r=1-\bar p_J+\bar p_J t$ and $s=1-\bar p_I+\bar p_I t$ gives
\[
\E[t^{S_M}]
=\bigl(w(1-\bar p_I+\bar p_I t)+(1-w)(1-\bar p_J+\bar p_J t)\bigr)^n
=(1-\bar p_N+\bar p_N\,t)^n,
\]
which is the PGF of $\mybin{n}{\bar p_N}$.
Since a distribution with finite support is characterized by its PGF,
$S_M\sim\mybin{n}{\bar p_N}$.
Similarly, if $U'_m\sim\mybin{m}{\bar q_I}$ and $V'_m\sim\mybin{n-m}{\bar q_J}$ are independent and
$S'_m:=U'_m+V'_m$, then the same calculation yields $S'_M\sim\mybin{n}{\bar q_N}$.
Therefore
\[
\delta_{N}=\TV(S_M,S'_M).
\]

For any event $A\subset\{0,1,\dots,n\}$,
\[
\P(S_M\in A)-\P(S'_M\in A)
=\E\!\left[\P(S_M\in A\mid M)-\P(S'_M\in A\mid M)\right]
\le \E\!\left[\TV(\Law(S_M\mid M),\Law(S'_M\mid M))\right],
\]
and taking the supremum over $A$ yields
\[
\TV(S_M,S'_M)\le \E\!\left[\TV(\Law(S_M\mid M),\Law(S'_M\mid M))\right]
=\sum_{m=0}^n \P(M=m)\,\TV(S_m,S'_m).
\]

Next, the map $(u,v)\mapsto u+v$ is a Markov kernel, and so
\[
\TV(S_m,S'_m)\le \TV\bigl((U_m,V_m),(U'_m,V'_m)\bigr).
\]
Moreover, for product measures one has
\begin{equation}
\label{eq:tv_product_bound}
\TV(P\times R,\;Q\times S)\le \TV(P,Q)+\TV(R,S),
\end{equation}
whence
\[
\TV(S_m,S'_m)\le \TV(U_m,U'_m)+\TV(V_m,V'_m).
\]
Define
\[
\tau_I(m):=\TV(\mybin{m}{\bar p_I},\mybin{m}{\bar q_I}),\qquad
\tau_J(k):=\TV(\mybin{k}{\bar p_J},\mybin{k}{\bar q_J}).
\]
Then
\[
\delta_{N}\le \E[\tau_I(M)]+\E[\tau_J(n-M)].
\]
It remains to estimate the two expectations.
Fix $\theta,\theta'\in[0,1]$ and define $f:\N\to\R$ by
\[
f(m):=\TV(\mybin{m}{\theta},\mybin{m}{\theta'}).
\]
We claim that $f$ is nondecreasing and subadditive.
To show subadditivity,
let $X\sim\mybin{m}{\theta}$, $Y\sim\mybin{k}{\theta}$ be independent so that
$X+Y\sim\mybin{m+k}{\theta}$, and similarly let $X'\sim\mybin{m}{\theta'}$,
$Y'\sim\mybin{k}{\theta'}$ independent so that $X'+Y'\sim\mybin{m+k}{\theta'}$.
Then
\[
f(m+k)=\TV(X+Y,\;X'+Y')\le \TV(X,X')+\TV(Y,Y')=f(m)+f(k).
\]
To show
monotonicity, observe that the following is a Markov kernel
that maps $\mybin{m+1}{\theta}$ to $\mybin{m}{\theta}$: interpret $\mybin{m+1}{\theta}$ as the number of successes in $m+1$ i.i.d.\ $\Ber(\theta)$ trials and delete one uniformly random trial.
Conditional on seeing $k$ successes among $m+1$ trials, the remaining number of successes equals $k$ with probability $\frac{m+1-k}{m+1}$ (deleted a failure) and equals $k-1$ with probability $\frac{k}{m+1}$ (deleted a success).
This kernel does not depend on $\theta$, and it sends $\mybin{m+1}{\theta}$ to $\mybin{m}{\theta}$ for every $\theta$.
Therefore, by TV contraction under the same kernel,
\[
f(m) = \TV(T_m(\mybin{m+1}{\theta}), T_m(\mybin{m+1}{\theta'})) \le \TV(\mybin{m+1}{\theta}, \mybin{m+1}{\theta'}) = f(m+1).
\]

Now let $m_0\ge 1$ be an integer. By subadditivity and monotonicity,
for every $m\ge 0$, writing $m=km_0+r$ with $0\le r<m_0$,
\[
f(m)\le kf(m_0)+f(r)\le (k+1)f(m_0)=\left\lceil\frac{m}{m_0}\right\rceil f(m_0).
\]
Hence, for any integer--valued $M$ with $\E M=m_0$, using $\lceil x\rceil\le x+1$,
\[
\E[f(M)]
\le f(m_0)\,\E\!\left[\left\lceil\frac{M}{m_0}\right\rceil\right]
\le f(m_0)\,\E\!\left[\frac{M}{m_0}+1\right]
=2f(m_0).
\]

Apply this with $f=\tau_I$ and $m_0=\nI$ (note $\E M=nw=\nI$) to get
\[
\E[\tau_I(M)]\le 2\tau_I(\nI)=2\delta_I,
\]
and with $f=\tau_J$ and $m_0=\nJ$ (since $\E(n-M)=\nJ$) to get
\[
\E[\tau_J(n-M)]\le 2\tau_J(\nJ)=2\delta_J.
\]
This completes the proof.
\end{proof}

\begin{appendix}

\section{Majorization kernels for complementary Bernoulli products}
\label{app:majorization}

Fix $n\ge2$ and write $\Omega:=\{0,1\}^n$.
For a parameter vector $\alpha=(\alpha_1,\dots,\alpha_n)\in[0,1]^n$ and $s\in\{+,-\}$ define the product measures
\[
\mu_{\alpha}^{\,s}
:=\bigotimes_{i=1}^n \Ber\!\left(\frac12+s\,\frac{\alpha_i}{2}\right).
\]
Thus $\mu_{\alpha}^{\,+}$ and $\mu_{\alpha}^{\,-}$ are complementary in the sense that
$\mu_{\alpha}^{\,-}$ is obtained from $\mu_{\alpha}^{\,+}$ by flipping each bit.

We show that, in this complementary setting, not only homogenization but any
\emph{majorization} of the parameter vector can be realized by a Markov kernel.
We write $\beta\prec\alpha$ to denote that $\beta$ is \emph{majorized} by $\alpha$, i.e., letting
$\alpha^{\downarrow}$ and $\beta^{\downarrow}$ be the vectors obtained by sorting the coordinates of
$\alpha$ and $\beta$ in nonincreasing order, we have
\[
\sum_{i=1}^k \beta_i^{\downarrow}\ \le\ \sum_{i=1}^k \alpha_i^{\downarrow}
\quad\text{for all }k=1,\dots,n-1,
\qquad\text{and}\qquad
\sum_{i=1}^n \beta_i^{\downarrow}\ =\ \sum_{i=1}^n \alpha_i^{\downarrow}.
\]
(see \cite[Ch.~1]{marshall2011inequalities}).
We thank L. Mattner \cite{Mattner2026} for this suggestion;
our original result had proven Lemma~\ref{lem:Ttransform-2coord}
only for $\lambda=\frac12$
and Theorem~\ref{thm:majorization-kernel} only for homogenization
via a compactness argument.

\subsection*{A two--coordinate kernel for a general $T$--transform}
The basic step is a Markov kernel acting on two coordinates which implements the
standard $T$--transform (a convex combination of the identity and a transposition).

\begin{lemma}[Two--coordinate $T$--transform kernel]
\label{lem:Ttransform-2coord}
Fix $a,b\in[0,1]$ and $\lambda\in[0,1]$, and define
\[
a' := \lambda a + (1-\lambda)b,
\qquad
b' := \lambda b + (1-\lambda)a.
\]
Let
\[
\eta := 
\begin{cases}
\displaystyle \frac{\lambda(1-\lambda)(a-b)^2}{2(1-ab)}, & ab<1,\\
[1ex]
0, & ab=1,
\end{cases}
\]
and consider the row--stochastic $4\times4$ matrix (indexed by $00,01,10,11$)
\begin{equation}
\label{eq:Mlambda}
M_{\lambda}(a,b):=
\left(\begin{array}{cccc}
1 & 0 & 0 & 0 \\
\eta & \lambda-\eta & (1-\lambda)-\eta & \eta \\
\eta & (1-\lambda)-\eta & \lambda-\eta & \eta \\
0 & 0 & 0 & 1 
\end{array}\right).
\end{equation}
Then all entries of $M_{\lambda}(a,b)$ are nonnegative, and for both $s\in\{+,-\}$,
\[
\Bigl(\Ber\!\left(\tfrac12+s\,\tfrac a2\right)\otimes \Ber\!\left(\tfrac12+s\,\tfrac b2\right)\Bigr)\,M_{\lambda}(a,b)
=
\Ber\!\left(\tfrac12+s\,\tfrac {a'}2\right)\otimes \Ber\!\left(\tfrac12+s\,\tfrac {b'}2\right).
\]
\end{lemma}

\begin{proof}
Row--stochasticity is immediate from \eqref{eq:Mlambda}.
For nonnegativity, note that for $a,b\in[0,1]$,
\[
2(1-ab)-(a-b)^2 = 2-a^2-b^2 \ge 0,
\]
hence $(a-b)^2\le 2(1-ab)$ whenever $ab<1$. Therefore
$\eta\le \lambda(1-\lambda)\le \min\{\lambda,1-\lambda\}$, which implies
$\lambda-\eta\ge0$ and $(1-\lambda)-\eta\ge0$.

To verify the mapping property, write (for $s\in\{+,-\}$)
$p_s:=\tfrac12+s\tfrac a2$, $q_s:=\tfrac12+s\tfrac b2$ and
$p_s':=\tfrac12+s\tfrac {a'}2$, $q_s':=\tfrac12+s\tfrac {b'}2$.
Let $x_s$ be the stochastic row vector of the input law
$\Ber(p_s)\otimes\Ber(q_s)$ ordered as $(00,01,10,11)$, and similarly let $y_s$
be that of $\Ber(p_s')\otimes\Ber(q_s')$.
A direct (routine) calculation using \eqref{eq:Mlambda} shows that
$x_s\,M_{\lambda}(a,b)=y_s$ for $s=+$ and for $s=-$. 
\end{proof}

\subsection*{Finite decomposition via majorization}
Lemma~\ref{lem:Ttransform-2coord} supplies the probabilistic implementation
of a single $T$--transform. The following theorem is the desired global statement.

\begin{theorem}[Majorization kernel for complementary products]
\label{thm:majorization-kernel}
Fix $n\ge2$ and $\alpha,\beta\in[0,1]^n$ with $\beta\prec\alpha$.
Then there exists a Markov kernel $K:\Omega\to\mathcal P(\Omega)$ such that
simultaneously, for both $s\in\{+,-\}$,
\[
\mu_{\alpha}^{\,s}K=\mu_{\beta}^{\,s}.
\]
\end{theorem}

\begin{proof}
By the Muirhead--Hardy--Littlewood--P\'olya lemma
(see \cite[Lemma~B.1]{marshall2011inequalities}),
there exist indices $(i_t,j_t)$ and parameters $\lambda_t\in[0,1]$, $t=1,\dots,m$,
such that if we define $\alpha^{(0)}:=\alpha$ and for each $t\ge1$ update only
coordinates $(i_t,j_t)$ via
\[
\bigl(\alpha^{(t)}_{i_t},\alpha^{(t)}_{j_t}\bigr)
=
\bigl(\lambda_t\alpha^{(t-1)}_{i_t}+(1-\lambda_t)\alpha^{(t-1)}_{j_t},\ 
\lambda_t\alpha^{(t-1)}_{j_t}+(1-\lambda_t)\alpha^{(t-1)}_{i_t}\bigr),
\]
with all other coordinates unchanged, then $\alpha^{(m)}=\beta$.

For each step $t$, let $K_t$ be the Markov kernel on $\Omega$ that acts only on
coordinates $(i_t,j_t)$ as follows: given $x\in\Omega$, sample
\[
(x'_{i_t},x'_{j_t})\sim M_{\lambda_t}\!\left(\alpha^{(t-1)}_{i_t},\alpha^{(t-1)}_{j_t}\right)(x_{i_t},x_{j_t})
\]
(using Lemma~\ref{lem:Ttransform-2coord}), and set $x'_k=x_k$ for $k\neq i_t,j_t$.
By independence of coordinates and Lemma~\ref{lem:Ttransform-2coord},
for both $s\in\{+,-\}$ we have
\[
\mu_{\alpha^{(t-1)}}^{\,s}K_t=\mu_{\alpha^{(t)}}^{\,s}.
\]
Finally, define the composed kernel $K:=K_1K_2\cdots K_m$. Iterating the last display yields
$\mu_{\alpha}^{\,s}K=\mu_{\beta}^{\,s}$ for both $s\in\{+,-\}$.
\end{proof}

\begin{corollary}[Homogenization as a special case]
\label{cor:homog-majorization}
Let $\alpha\in[0,1]^n$ and put $\bar\alpha:=\frac1n\sum_{i=1}^n \alpha_i$.
Then there exists a Markov kernel $K$ on $\Omega$ such that for both $s\in\{+,-\}$,
\[
\mu_{\alpha}^{\,s}K=\mu_{\bar\alpha\mathbf 1}^{\,s}.
\]
\end{corollary}

\begin{proof}
It is immediate that $\bar\alpha\mathbf 1 \prec \alpha$. Apply
Theorem~\ref{thm:majorization-kernel} with $\beta=\bar\alpha\mathbf 1$.
\end{proof}

\end{appendix}

\paragraph{Acknowledgments}
Some of the proofs were assisted by ChatGPT Pro.
We thank Sébastien Bubeck, Mark Sellke and Nikita Zhivotovskiy
for pointing out errors in earlier versions.
We are grateful to Lutz Mattner and Bero Roos
for detailed and insightful comments on the manuscript.

This research was supported in part by 
the Israel Science Foundation ISF grant
581/25
and the Binational Science Foundation BSF grant
2024243.

\end{document}